\documentclass{amsart}
\usepackage{amssymb,latexsym}
\usepackage{eepic,epic}
\input{epsf}
\topskip
\begin{document}

\pagenumbering{arabic} \pagestyle{headings}

\title{On representable graphs, semi-transitive orientations, and the representation numbers}
\maketitle

\begin{center}
{\bf Magn\'us M. Halld\'orsson}\footnote{The work presented here was supported by grant no. 070009022 from the Icelandic Research Fund.}\\
\emph{School of Computer Science, Reykjavik University, Kringlan 1,
103 Reykjavik, Iceland}\\
{\tt mmh@ru.is}\\
\vskip 10pt
{\bf Sergey Kitaev\footnote{The work presented here was supported by grant no. 060005012/3 from the Icelandic Research Fund.}}\\
\emph{The Mathematics Institute, Reykjavik University, Kringlan 1,
103 Reykjavik, Iceland}\\
{\tt sergey@ru.is}\\
\vskip 10pt
{\bf Artem Pyatkin\footnote{The
work was partially supported by grants of the Russian Foundation for
Basic Research (projects codes 07-07-00022 and 08-01-00516)}}\\
\emph{Sobolev Institute of Mathematics, pr-t Koptyuga 4, 630090,
Novosibirsk, Russia}\\
{\tt artem@math.nsc.ru}\\
\end{center}

\section*{Abstract}

A graph $G=(V,E)$ is representable if there exists a word $W$ over
the alphabet $V$ such that letters $x$ and $y$ alternate in $W$ if
and only if $(x,y)\in E$ for each $x\neq y$. If $W$ is $k$-uniform
(each letter of $W$ occurs exactly $k$ times in it) then $G$ is
called $k$-representable. It was shown in~\cite{KP} that a graph is
representable if and only if it is $k$-representable for some $k$.
Minimum $k$ for which a representable graph $G$ is $k$-representable
is called its representation number.

In this paper we give a characterization of representable graphs in
terms of orientations. Namely, we show that a graph is representable if
and only if it admits an orientation into a so-called
\emph{semi-transitive digraph}.
This allows us to prove a number of results about representable graphs,
not the least that 3-colorable graphs are representable.
We also prove that the representation number of a graph on
$n$ nodes is at most $n$, from which one concludes that the
recognition problem for representable graphs is in NP.
This bound is tight up to a constant factor, as we present a graph
whose representation number is $n/2$.

We also answer several questions posed in~\cite{KP}, in particular,
on representability of the Petersen graph and local permutation
representability.



{\bf Keywords:} graph, representation, words, orientations,
complexity, circle graph, 3-colorable graph, comparability graph,
Petersen graph, representation number, semi-transitive orientation


\newcommand{\psdrawa}[3]{\begin{array}{c} \hspace{-1mm}
\raisebox{-4pt}{\psfig{figure=#1.ps,width=#2,height=#3}}
\hspace{0pt}\end{array}}

\newtheorem{prop}{Proposition}
\newtheorem{lemma}[prop]{Lemma}
\newtheorem{obs}[prop]{Observation}
\newtheorem{cor}[prop]{Corollary}
\newtheorem{theorem}[prop]{Theorem}
\newtheorem{problem}{Problem}
\newtheorem{conj}{Conjecture}
\theoremstyle{definition}
\newtheorem{defin}[prop]{Definition}
\newtheorem{claim}[prop]{Claim}
\newtheorem{ex}{Example}
\newtheorem{remark}[prop]{Remark}

\def\R{\mathcal{R}}
\newcommand\p{\circle*{0.3}}

\section{Introduction}
A graph $G=(V,E)$ is {\em representable} if there exists a word $W$
over the alphabet $V$ such that letters $x$ and $y$ alternate in $W$
if and only if $(x,y)\in E$ for each $x\neq y$. It is
$k$-representable if each letter appears exactly $k$ times. The
notion of representable (directed) graphs was introduced
in~\cite{KS} to obtain asymptotic bounds on the {\em free spectrum}
of the widely-studied {\em Perkins semigroup} which has played
central role in semigroup theory since 1960, particularly as a
source of examples and counterexamples. In \cite{KP}, the only paper
solely dedicated to the study of representable graphs, numerous
properties of representable graphs are derived and numerous types of
representable and non-representable graphs are pinpointed. Still,
large gaps of knowledge of these graphs have remained, and the
purpose of this paper is to address them.

We address the three most fundamental issues about representable graphs:
\begin{itemize}
  \item Are there alternative representations of these graphs that
aid in reasoning about their properties?
 \item Which types of graphs are representable and which ones are not? And,
 \item How large words can be needed to represent representable graphs?
\end{itemize}
These can be viewed as some of the most basic questions of any graph
class. We make progress on each of these.

We characterize representable graphs in terms of orientability.
The edges can be directed so as to yield a directed graphs satisfying a
property that we call \emph{semi-transitivity}.
It properly generalizes the transitivity property of comparability
graphs, constraining the subgraphs induced by certain types of cycles.
The definition and the characterization is given in Section 3.
This formulation allows us to reason fairly easily about the types of
graphs that are representable.

We show that the class of representable graphs captures quite involved
properties. In particular, all 3-colorable graphs are representable.
This resolves a conjecture of \cite{KP} regarding the Petersen graph,
showing that it is representable. We actually give an explicit
construction to show that it is 3-representable. The result also
properly captures all the previously known classes of representable
graphs: outerplanar, prisms, and comparability graphs.
On the negative side, we answer an open question of \cite{KP} by
presenting a non-representable graphs all of whose induced
neighborhoods are comparability graphs.

Finally, we show that any representable graph is $n$-representable, again
utilizing the semi-transitive orientability. Previously, no non-trivial
upper bound was known on the representation number, which is the
smallest value $k$ such that the given graph is $k$-representable.
This result implies that problem of deciding whether a given graph is
representable is contained in NP.
This bound on the representation number is tight up to a constant
factor, as we construct graphs with representation number $n/2$.
We also show that deciding if a representable graph is
$k$-representable is NP-complete for $3 \le k \le \lceil n/2
\rceil$, while the class of circle graphs coincides with the class
of graphs with representation number at most 2.

The paper is organized as follows. In Section~\ref{def} we give
definitions of objects of interest and review some of the known
results.
In Section~\ref{char} we give a characterization of representable
graphs in terms of orientations and discuss some important
corollaries of this fact. In Section~\ref{nice} we consider the
problems concerning the representation numbers, and show that it is
always at most $n$ but can be as much as $n/2$. We explore in
Section~\ref{classif} which classes of graphs are representable,
showing, in particular, 3-colorable graphs to be representable, but
numerous others to be orthogonal to representability.
Finally, we conclude with a discussion of
algorithmic complexity and some open problems in Section~\ref{open}.

\section{Definitions, notation, and known results}\label{def}

In this section we follow~\cite{KP} to define the objects of
interest.

Let $W$ be a finite word over an alphabet $\{x_1, x_2, \ldots \}$.
If $W$ involves the letters $x_1, x_2, \ldots, x_n$ then  we write
$Var(W) = \{x_1, \ldots, x_n\}$. Let $X$ be a subset of $Var(W)$.
Then $W\setminus X$ is the word obtained by eliminating all letters
in $X$ from $W$. A word is $k$-\emph{uniform} if each letter appears
in it exactly $k$ times. A 1-uniform word is also called a {\it
permutation}. Denote by $W_1W_2$ the {\em concatenation} of words
$W_1$ and $W_2$. We say that the letters $x_i$ and $x_j$ {\it
alternate} in $W$ if the word induced by these two letters contains
neither $x_ix_i$ nor $x_jx_j$ as a factor. If a word $W$ contains
$k$ copies of a letter $x$ then we denote these $k$ appearances of
$x$ by $x^1,x^2,\ldots ,x^k$. We write $x_i^j<x_k^l$ if $x_i^j$
stays in $W$ before $x_k^l$, i.~e., $x_i^j$ is to the left of
$x_k^l$ in $W$.

Let $G=(V,E)$ be a graph with the vertex set $V$ and the edge
set~$E$. We say that a word $W$ \emph{represents} the graph $G$ if
there is a bijection $\phi:Var(W)\rightarrow V$ such that
$(\phi(x_i),\phi(x_j))\in E$ if and only if $x_i$ and $x_j$
alternate in $W$. It is convenient to identify the vertices of a
representable graph and the corresponding letters of a word
representing it. We call a graph $G$ \emph{representable} if there
exists a word $W$ that represents $G$.
If $G$ can be represented by a
$k$-uniform word, then we say that $G$ is \emph{$k$-representable}.
Clearly, the complete graphs are the only examples of
1-representable graphs. So, in what follows we assume that $k\ge 2$.
Let the
\emph{representation number} of a graph $G$ be the minimum $k$ such that
$G$ is $k$-representable.


We call a graph \emph{permutationally representable} if it can be
represented by a word of the form $P_1P_2\ldots P_k$ where all $P_i$
are permutations. In particular, all permutationally representable
graphs are $k$-representable.

A digraph (directed graph) is {\em transitive} if the adjacency relation is transitive, i.~e. for every vertices $x,y,z\in V$ the existence of the arcs $xy,yz\in E$ yields that
$xz\in E$. A {\em comparability graph} is an undirected graph having an
orientation of the edges that yields a transitive digraph.


The following results on representable graphs are known. They were proved
in~\cite{KP} except for the Lemma~\ref{l1} that was proved in \cite{KS}.

\begin{prop}\label{pr1}
Let $W=AB$ be a $k$-uniform word representing a graph $G$.
Then the word $W'=BA$ also $k$-represents $G$.
\end{prop}

\begin{prop}\label{pr1a}
Let the graphs $G_1$ and $G_2$ be $k$-representable and $x\in V(G_1), y\in V(G_2)$. Assume that the graph $G$ is obtained
from $G_1$ and $G_2$ by identifying the
vertices $x$ and $y$ into a new vertex~$z$. Then $G$ is also $k$-representable.
\end{prop}


\begin{lemma}\label{l1}
A graph is permutationally representable if and only if it is a
comparability graph. In particular, all bipartite graphs are
permutationally representable.
\end{lemma}

%
For a vertex $x\in V(G)$ denote by $N(x)$ the set of all its
neighbors.

\begin{theorem}~\label{t1}
If $G$ is representable, then for every $x\in V(G)$ the graph
induced by $N(x)$ is permutationally representable.
\end{theorem}

\begin{theorem}~\label{t2}
Outerplanar graphs are $2$-representable.
\end{theorem}



\begin{theorem}~\label{t3}
For every graph $G$ there exists a $3$-representable graph $H$
that contains $G$ as a minor. In particular, a $3$-subdivision of
every graph $G$ is $3$-representable.
\end{theorem}

\begin{prop}\label{pr4}
All prisms are $3$-representable. Moreover, the triangular prism has
the representation number~3.
 \end{prop}


Theorem~\ref{t1} provides an easy way to construct non-representable
graphs: just take a graph that is not a comparability graph and add
an all-adjacent vertex to it. The wheel $W_5$ is the smallest
non-representable graph. Some other small non-representative graphs
are given in Fig.~1 (this figure appears in~\cite{KP}).

\begin{figure}[ht]
\begin{center}
\input{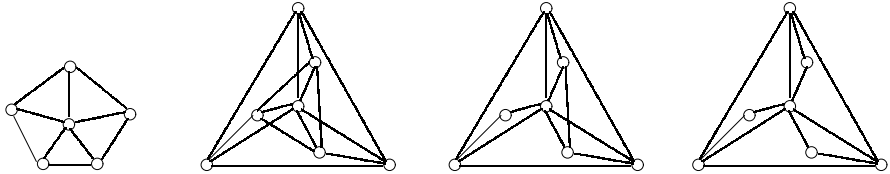}
\vspace{-20mm}\caption{Small non-representable graphs}\label{small}
\end{center}
\end{figure}

Paper~\cite{KP} contains several open problems. In this paper we
solve some of them.


\section{Characterization of Representable Graphs by Orientability}\label{char}

The word representation of representable graphs is simple and natural.
Yet it does not lend to easy arguments for the characteristic of
representable graphs. Non-representability is even harder to argue in
terms of the many possible corresponding words.
The main result of this section is a new characterization of
representable graphs that leads easily to various results about
representability.

We give a characterization in terms of orientability, which implies
that representability corresponds to a property of a digraph
obtained by directing the edges in certain way. Recall that
Lemma~\ref{l1} states that a graph is permutationally representable
if and only if it has a transitive orientation. We prove a similar
fact on representable graphs, namely, that a graph is representable
if and only if it has a so-called {\em semi-transitive orientation}.
Our definition, in fact, generalizes
that of a transitive orientation.

Other orientations have been defined in order to capture
generalizations of comparability graphs. As transitive orientations
form constraints on the orderings of induced $P_3$, these
generalizations form constraints on the orderings of induced $P_4$.
These include {\em perfectly orderable graphs} (and its subclasses)
and {\em opposition graphs}~\cite{classes}.
Classes such as {\em chordal graphs} are defined in terms of
vertex-orderings, and imply therefore indirectly acyclic
orientations. None of these properties captures our definition
below, nor does our characterization subsume any of them.

We turn to the characterization and start with definitions of
certain directed graphs. A \emph{semi-cycle} is the directed acyclic
graph obtained by reversing the direction of one arc of a directed
cycle. An acyclic digraph is a \emph{shortcut} if it is induced by
the vertices of a semi-cycle and contains a pair of non-adjacent
vertices. Thus, a digraph on the vertex set $\{ v_0, v_1, \ldots,
v_t\}$, is a shortcut if it contains a directed path $v_0v_1\ldots
v_t$, the arc $v_0v_t$ and it is missing the arcs $v_iv_j$ for $0
\le i < j \le t$ (in particular, $t \ge 3$).

\begin{defin}
A digraph is \emph{semi-transitive} if it is acyclic and contains no
shortcuts.
\end{defin}

A graph is \emph{semi-transitively orientable} if there exists an
orientation of the edges that results in a semi-transitive graph.

Our main result in this paper is the following.

\begin{theorem}
A graph is representable if and only if it is semi-transitively orientable.
\label{thm:rep-equals-semi-trans}
\end{theorem}

We first need some additional definitions and lemmas. A {\em
topological order} (or {\em topsort}) of an acyclic digraph is a
permutation of the vertices that obeys the arcs, i.~e. for each arc
$uv$, $u$ precedes $v$ in the permutation. For a node-labeled
digraph, let it also refer to the word obtained by visiting the
nodes in that order. Let $D=(V,E)$ be a digraph.  The
\emph{$t$-string} digraph $D^t$ of $D$ is defined as follows.  The
vertices of $D^t$ are $v^i$, for $v \in V$ and $i=1, 2, \ldots, t$,
and $v^iu^j$ is an arc in $D^t$ if and only if either $i=j$ and $vu
\in E$ or $i < j$ and $uv \in E$. Intuitively, the $t$-string
digraph of $D$ has $t$ copies of $D$ strung together. Given a word
$S$, let $G_S$ denote the graph represented by $S$.  If $S$ is a
topsort of $D^t$ then we also denote by $G_S$ the graph represented
by the word $S'$ obtained from $S$ by omitting the superindices of
the vertices (i.~e. the copies of the same vertex in $S$ are
considered as the same letters in $S'$).

Given a digraph $D$, let $G_D$ be the graph obtained by ignoring orientation.

We argue that the word representing a semi-transitive digraph comes
from a special topological ordering of the $t$-string digraph $D^t$
for some $t$. We first observe that any topological ordering of
$D^t$ preserves arcs.

\begin{lemma}\label{lem-top-preserves}
Let $D$ be a digraph with distinct node-labels.
Let $S$ be
a topological ordering of a $D^t$.
Then $G_D$ is a subgraph of $G_S$.
\end{lemma}

\begin{proof}
Consider an edge $uv$ in $G_D$, and suppose without loss of
generality that it is directed as $uv$ in $D$. Then, in $D^t$, there
is a directed path $u^1v^1u^2v^2 \ldots u^tv^t$. Thus, occurrences
of $u$ and $v$ in a topsort of $D^t$ are alternating. Hence,
$uv \in G_S$.
\end{proof}

To prove equivalence, we now give a method to produce a topological
ordering that generates all non-arcs. We say that a subgraph $H$
\emph{covers} a set $A$ of non-arcs if each non-arc in $A$ is also
found in $H$. A word covers the non-arcs if the graph it represents
covers them.

\begin{lemma}
The non-arcs incident with a path in a semi-transitive digraph can
be covered with a $2$-uniform word. \label{lem:path-to-2rep}
\end{lemma}

\begin{proof}
Let $P$ be a path in a semi-transitive digraph $D$. We shall form a
topsort $S$ of the 2-string digraph $D^2$ and show that it covers
all non-arcs having at least one endpoint on $P$.

We say that a node $x$ of $D^{2}$ \emph{depends on} node $y$ if there is a directed path from $y$ to $x$ in $D^{2}$, i.~e.
$y$ must
appear before $x$ in a topological ordering of $D^{2}$. We use the notation $y \leadsto x$ if $x$ depends on $y$.
A node is listed \emph{earliest possible} if it is listed as soon as
all nodes that it depends on have been listed.
A node is listed \emph{latest possible} if it is listed after all nodes
that do not depend on it.

Let $S$ be any topological ordering of $D^2$ where the first occurrences
of nodes in $P$ are as late as possible and the second occurrences are
as early as possible. The ordering of other nodes is arbitrary, within
these constraints.

We claim that this word $S$ covers all non-arcs involving nodes in
$P$. Consider a pair $u,v$, where $uv \not\in G_D$ and $u \in P$.
Note that $v$ may also belong to $P$, in which case we may assume
that the path goes from $u$ to $v$. Consider the listings of $u^1,
v^1, u^2, v^2$, where the subscript refers to the occurrence number
of the node. Observe that $u$ may depend on $v$, or vice versa, but
not both. There are three cases to consider.

Case (i): There is a path from $u$ to $v$ in $D$. We claim that
$u^2$ does not depend on $v^1$. Suppose it does, i.~e. $v^1 \leadsto
u^2$. Then, there is an arc $x^1y^2 \in D^2$ such that $v^1 \leadsto
x^1$  and $y^2 \leadsto u^2$. By the assumptions and the symmetry of
the two copies of $D$ in $D^2$, it follows that $y^1 \leadsto u^1
\leadsto v^1 \leadsto x^1$. By the definition of $2$-string graphs,
$yx$ is an arc in $D$, so $y^1x^1 \in E(D^2)$. Then, by
semi-transitivity, $u^1v^1 \in E(D^2)$, or $uv \in E(G_D)$, which is
a contradiction. It now follows that the nodes will occur as $u^1
u^2 v^1 v^2$ in $S$, i.~e. $uv\not\in E(G_S)$.

Case (ii): There is a path from $v$ to $u$ in $D$.
This is symmetric to case (i), with $u$ replaced by $v$.
Thus, the nodes will occur as $v^1 v^2 u^1 u^2$ in $S$.

Case (iii): The nodes $u$ and $v$ are incomparable in $D$. In
particular, $v$ is not in $P$. Then, $u^1$ and $v^1$ do not depend
on each other, nor do $u^2$ and $v^2$. If $v^2$ depends on $u^1$
then the nodes occur as $v^1u^1u^2v^2$ in $S$. Otherwise, their
order is $v^1v^2u^1u^2$.
\end{proof}


We now return to the proof of
Theorem~\ref{thm:rep-equals-semi-trans}, starting with the forward
direction. Given a word-representant $S$, we direct an edge from $x$
to $y$ if the first occurrence of $x$ is before that of $y$ in the
word. Let us show that such orientation $D$ of $G_S$ is
semi-transitive. Indeed, assume that $x_0x_t\in E(D)$ and there is a
directed path $x_0x_1\ldots x_t$ in $D$. Then in the word $S$ we
have $x_0^i<x_1^i<\ldots<x_t^i$ for every $i$. Since $x_0x_t\in
E(D)$ we have $x_t^i<x_0^{i+1}$. But then for every $j<k$ and $i$
there must be $x_j^i<x_k^i<x_j^{i+1}$, i.~e. $x_ix_j\in E(D)$. So,
$D$ is semi-transitive.


For the other direction, denote by $G$ the graph and by $D$ its
semi-transitive orientation.  Let
$P_1,P_2,\ldots, P_{\tau}$ be the set of directed paths covering all
vertices of $D$. For every $i=1,2,\ldots, \tau$ denote by $S_i$ the
topsort of the digraph $D^2$ satisfying the conditions of
Lemma~\ref{lem:path-to-2rep} for the path $P_i$. Put $S=S_1S_2\ldots
S_{\tau}$. Clearly, $S$ is a $2\tau$-uniform word; it can be treated as
a topsort of a $2\tau$-string $D^{2\tau}$. Then $G=G_S$. Indeed, by
Lemma~\ref{lem-top-preserves} we have $E(G)\subset E(G_S)$. On the other
hand, if $uv\not\in E(G)$ then $u\in P_i$ for some $i$, and thus by
Lemma~\ref{lem:path-to-2rep} the letters $u$ and $v$ are not alternating
in the subword $S_i$. Therefore, $uv\not\in E(S)$.
Theorem~\ref{thm:rep-equals-semi-trans} is proved. $\qed$
\smallskip

Theorem~\ref{thm:rep-equals-semi-trans} makes clear the relationship to
comparability graphs, which are those that have transitive orientations.
Since transitive digraphs are also semi-transitive, this immediately
implies that comparability graphs are representable.

\section{The Representation Number of Graphs}\label{nice}

We focus now on the following question: Given a representable graph,
how large is its representation number? In~\cite{KP}, certain
classes of graphs were proved to be 2- or 3-representable, and an
example was given of a graph (the triangular prism) with the
representation number of 3. On the other hand, no examples were
known of graphs with representation numbers larger than 3, nor were
there any non-trivial upper bounds known. We show here that the
maximum representation number of representable graphs is linear in
the number of vertices.

For the upper bound, we use the results of the preceding section.
We have the following directly from the proof
of Theorem~\ref{thm:rep-equals-semi-trans}.

\begin{cor} \label{cor1}
A representative graph $G$ is $2\tau(G)$-representable, where
$\tau(G)$ is the minimum number of paths covering all nodes in some
semi-transitive orientation of $G$.
\end{cor}

This immediately gives an upper bound of $2n$ on the representation number.
We can improve this somewhat with an effective procedure.

\begin{theorem} \label{thm:di-to-string}
Given a semi-transitive digraph $D$ on $n$ vertices, there is a
polynomial time algorithm that generates an $n$-uniform word
representing $G_D$. Thus, each representable graph is $n$-representable.
\end{theorem}

\begin{proof} The algorithm works as follows.

Step 0. Start with $A=\emptyset$ and $i=1$.

Step $i$. If $D$ contains a path $P_i$ covering at least two vertices from $V\setminus A$ then let $A:=A\cup V(P_i)$ and $i:=i+1$. Otherwise,
let $B=V\setminus A$ and go to the Final Step.

Final Step. Let $S_i$ be the topsort of the digraph $D^2$ satisfying the
conditions of Lemma~\ref{lem:path-to-2rep} for the path $P_i$ and put
$S'=S_1S_2\ldots S_t$ where $t$ is the number of paths found at previous
steps. If $|B|\le 1$ then let $S=S'$. Otherwise, consider a topsort
$S_0$ of $D$ where the vertices of $B$ are listed in a row (since the
vertices of $B$ do not depend on each other, such a topsort must exist)
and in particular in the reverse order of their appearance in $S_1$.
Let $S = S' S_0$.

Clearly, $G_D=G_S$ (the proof is the same as in
Theorem~\ref{thm:rep-equals-semi-trans}).
It is easy to verify that each letter appears in $S$ at most $n$ times.
\end{proof}

Theorem~\ref{thm:di-to-string} implies that the graph
representability is polynomially verifiable, answering an open
question in~\cite{KP}. Indeed, having a representable graph $G$, we
may ask for a word $W$ $k$-representing it and verify this fact in
time bounded by the polynomial on $k$ and $n$. Since $k\le n$, this
is a polynomial on $n$. So, we have proved
\begin{cor}
The recognition problem for representable graphs is in NP.
\label{cor:inNP}
\end{cor}





We now show that there are graphs with representation number of $n/2$,
matching the upper bound within a factor of 2.

The {\em cocktail party graph} $H_{k,k}$ is the
graph obtained from the complete bipartite graph $K_{k,k}$ by
removing a perfect matching. Denote by $G_k$ the graph obtained from a
cocktail party graph $H_{k,k}$ by adding an all-adjacent vertex.

\begin{theorem} \label{example}
The graph $G_k$ has representation number $k=\lfloor n/2\rfloor$.
\end{theorem}

The proof is based on three statements.

\begin{lemma}
Let $H$ be a graph and $G$ be the graph obtained from $H$ by adding an
all-adjacent vertex.
Then $G$ is $k$-representable if and only if $H$ is permutationally $k$-representable.
\label{obs:rep-permrep}
\end{lemma}

\begin{proof}
Let 0 be the letter corresponding to the all-adjacent vertex. Then
every other letter of the word $W$ representing $G$ must appear
exactly once between two consecutive zeroes. We may assume also that
$W$ starts with 0. Then the word $W\setminus \{0\}$ is a
permutational $k$-representation of $H$. Conversely, if $W'$ is a
word permutationally $k$-representing $H$, then we insert 0 in front
of each permutation to get a $k$-representation (in fact
permutational) of $G$.
\end{proof}

Recall that the {\em dimension} of a poset is the minimum
number of linear orders such that their intersection induces this
poset.

\begin{lemma}
A comparability graph is permutationally
$k$-representable if and only if the poset induced by this graph
has dimension at most $k$.
\label{obs:dimension}
\end{lemma}

\begin{proof}
Let $H$ be a comparability graph and $W$ be a word permutationally
$k$-representing it.  Each permutation in $W$ can be considered as a
linear order where $a<b$ if $a$ meets before $b$ in the permutation (and
vice versa). We want to show that the comparability graph of the poset
induced by the intersection of these linear orders coincides with $H$.

Two vertices $a$ and $b$ are adjacent in $H$ if and only if their
letters alternate in the word. So, they must be in the same order in
each permutation, i.~e. either $a<b$ in every linear order or $b<a$
in every linear order. But this means that $a$ and $b$ are comparable in the
poset induced by the intersection of the linear orders, i.~e. $a$ and
$b$ are adjacent in its comparability graph.
\end{proof}

The next statement most probably is known but we give its proof here for
the sake of completeness.

\begin{lemma} \label{cocktail}
The poset $P$ over $2k$
elements $\{a_1,a_2,\ldots,a_k,b_1,b_2,\ldots,b_k\}$ such that $a_i<b_j$ for
every $i\ne j$ and all other elements are not comparable has dimension $k$.
\end{lemma}

\begin{proof}
Assume that this poset is the intersection of $t$ linear orders.
Since $a_i$ and $b_i$ are not comparable for each $i$, their must be
a linear order where $b_i<a_i$. If we have in some linear order both
$b_i<a_i$ and $b_j<a_j$ for $i\ne j$, then either $a_i<a_j$
or $a_j<a_i$ in it. In the first case we have that $b_i<a_j$, in the
second that $b_j<a_i$. But each of these inequalities contradicts
the definition of the poset. Therefore, $t\ge k$.

In order to show that $t=k$ we can consider a linear order
$a_1<a_2<\ldots<a_{k-1}<b_k<a_k<b_{k-1}<\ldots<b_2<b_1$ together
with all linear orders obtained from this order by the simultaneous
exchange of $a_k$ and $b_k$ with $a_m$ and $b_m$ respectively
($m=1,2,\ldots,k-1$). It can be verified that the intersection of
these $k$ linear orders coincides with our poset.
\end{proof}

Now we can prove Theorem~\ref{example}. Since the cocktail party graph
$H_{k,k}$ is a comparability graph of the poset $P$, we deduce from
Lemmas~\ref{cocktail} and~\ref{obs:dimension} that $H_{k,k}$ is
permutationally $k$-representable but not permutationally
$(k-1)$-representable. Then by Lemma~\ref{obs:rep-permrep} we have that
$G_k$ is $k$-representable but not $(k-1)$-representable.
Theorem~\ref{example} is proved. \qed
\medskip

The above arguments help us also in deciding the complexity of
determinining the representation number.
From Lemmas \ref{obs:rep-permrep} and \ref{obs:dimension}, we see
that it is as hard as determining
the dimension $k$ of a poset. Yannakakis \cite{Yann} showed that the latter
is NP-hard, for any $3 \le k \leq \lceil n/2\rceil$.
We therefore obtain the following.

\begin{prop}
Deciding whether a given graph is
$k$-representable, for any given $3\leq k\leq\lceil n/2 \rceil$,
is NP-complete.
\end{prop}

It was further shown by
Hegde and Jain \cite{HJ} that it is NP-hard to
approximate the dimension of a poset within almost a square root factor.
We therefore obtain the following hardness for the representation number.

\begin{prop}
Approximating the representation number within $n^{1/2-\epsilon}$-factor is
NP-hard, for any $\epsilon > 0$.
\end{prop}

In contrast with these hardness results, the case $k=2$ turns out to be
easier and admits a succinct characterization.
The following fact essentially appears in~\cite{C}.
Recall that a graph is called {\em a circle graph} if we can arrange
its vertices as chords on a circle in such a way that two nodes in
the graph are adjacent if and only if the corresponding chords
overlap.

\begin{obs}\label{2repr}
A graph is 2-representable if and only if it is a circle graph.
\end{obs}

Indeed, given a circle graph $G$, consider the ends of the chords on a
circle as a letters and read the obtained word in a clockwise order
starting from an arbitrary point. It is easy to see that two chords
intersect if and only if the corresponding letters alternate in the
word.  For the opposite direction, place $2n$ nodes at a circle in the
same order as they meet in the word and connect the same letters by
chords.

It follows from Theorems~\ref{2repr} and~\ref{t2} that outerplanar
graphs are circle graphs. Theorem~\ref{2repr} can also be useful as
a tool in proving that a graph is not a circle graph. For example,
non-representable graphs (for instance, all odd wheels $W_{2t+1}$
for $t\geq 2$) are not circle graphs.


\section{Characteristics of Representable Graphs}\label{classif}

When faced with a new graph class, the most basic questions involve the
kind of properties it satisfies: which known classes are properly
contained (and which not), which graphs are otherwise contained (and
which not), what operations preserve representability (or
non-representability), and which properties hold for these graphs.

Previously, it was known that the class of representable graphs
includes comparability graphs, outerplanar graphs, subdivision
graphs, and prisms. The purpose of this section is to clarify this
situation significantly, including resolving some conjectures. We
start with exploring the impact of colorability on representability.

\paragraph{\bf Chromatic number and representability}


\begin{theorem} \label{obs:3col}
3-colorable graphs are semi-transitive, and thus representable.
\end{theorem}

\begin{proof}
Given a 3-coloring of a graph,
direct its edges from the first color class through the second to
the third class. It is easy to see that we obtain a semi-transitive digraph.
\end{proof}

This implies a number of earlier results on representability,
including that of outerplanar graphs, subdivision graphs, and prisms
(see Theorems~\ref{t2},~\ref{t3} and Proposition~\ref{pr4}). The
theorem also shows that
2-degenerate are representable, as well as graphs of maximum degree
3 (via Brooks theorem).

This result does not extend to higher chromatic numbers. The
examples in Fig.~\ref{small} show that 4-colorable graphs can be
non-representable. We can, however, obtain a results in terms of the
{\em girth} of the graph, which is the length of its shortest cycle.

\begin{prop}
Let $G$ be a graph whose girth is greater than its chromatic number.
Then, $G$ is representable.
\end{prop}

\begin{proof}
Suppose the graph is colored with $\chi(G)$ natural numbers. Orient
the edges of the graph from small to large colors. There is no
directed path with more than $\chi(G)-1$ arcs, but since $G$
contains no cycle of $\chi(G)$ or fewer arcs, there can be no
shortcut. Hence, the digraph is semi-transitive.
\end{proof}

Theorem \ref{obs:3col} also implies that the Petersen graph is
representable, turning down a conjecture in ~\cite{KP}. We can show
that the graph is actually 3-representable. We give here two of its
3-representations, related to the numbering in Fig.~\ref{petersen},
that were found in~\cite{KL}:

\begin{itemize}
\item 1, 3, 8, 7, 2, 9, 6, 10, 7, 4, 9, 3, 5, 4, 1, 2, 8, 3, 10, 7,
6, 8, 5, 10, 1, 9, 4, 5, 6, 2 \item 1, 3, 4, 10, 5, 8, 6, 7, 9, 10,
2, 7, 3, 4, 1, 2, 8, 3, 5, 10, 6, 8, 1, 9, 7, 2, 6, 4, 9, 5
\end{itemize}

\setlength{\unitlength}{7mm}
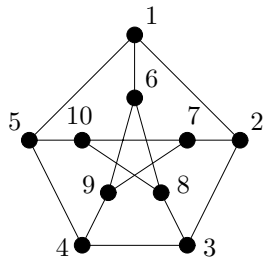
\begin{figure}[h]
\begin{center}
\begin{picture}(8,4.5) \put(2,-1){
\put(2,5){\p}\put(2.2,5.2){1}
\put(2,3.8){\p}\put(2.2,4){6}\put(-0.4,3.2){5}\put(0,3){\p}
\put(0.7,3.3){10}\put(1,3){\p} \put(3,3.3){7}\put(3,3){\p}
\put(4.2,3.2){2}\put(4,3){\p}
\put(0.5,0.8){4}\put(1,1){\p}\put(3.3,0.8){3}\put(3,1){\p}
\put(1,2){9}\put(1.5,2){\p}\put(2.8,2){8}\put(2.5,2){\p}

\path(2,5)(2,3.8)(1.5,2)(3,3)(1,3)(0,3)(1,1)(3,1)(4,3)(2,5)
\path(1,3)(2.5,2)(2,3.8)\path(2,5)(0,3)\path(1,1)(1.5,2)\path(2.5,2)(3,1)\path(3,3)(4,3)}\end{picture}
\caption{Petersen's graph}\label{petersen}
\end{center}
\end{figure}

The following argument shows that Petersen's graph is {\em not}
2-representable. Suppose that the graph is 2-representable and $W$ is a word
2-representing it. Let $x$ be a letter in $W$ such that there are
minimum number of letters between the two appearances of $x$. Clearly,
there are exactly three different letters between them. By symmetry, we
can assume that $x=1$
and by Proposition~\ref{pr1} we can assume that $W$ starts with 1.
So, letters 2,5, and 6 are between the two 1's and because of
symmetry, the fact that Petersen's graph is edge-transitive (that
is, each of its edges can be made ``internal''),
and taking into account that nodes 2, 5, and 6 are pairwise not adjacent,
we can assume that
$W=12561W_16W_25W_32W_4$ where $W_i$'s are some factors for
$i=1,2,3,4$. To alternate with 6 but not to alternate with 5, we
must have $8\in W_1$ and $8\in W_2$. Also, to alternate with 2 but
not to alternate with 5, we must have $3\in W_3$ and $3\in W_4$. But
then 8833 is a subsequence in $W$ and thus 8 and 3 are not adjacent in
the graph, a contradiction.

We explore now further graph properties that are orthogonal to
representability.

\paragraph{\bf Non-representable graphs}

One of the open problems posed in \cite{KP} was the following.

\begin{problem}~\label{p1}
Are there any non-representable graphs that do not satisfy the
conditions of Theorem~\ref{t1}?
\end{problem}

We give here the positive answer.  A counterexample to the converse
of Theorem~\ref{t1} is given by the graph in Fig.~\ref{co-T2} called
co-($T_2$) in~\cite{LINK}. It is easy to check that the induced
neighborhood of any node of the graph co-$(T_2)$ is a comparability
graph.

\setlength{\unitlength}{7mm}
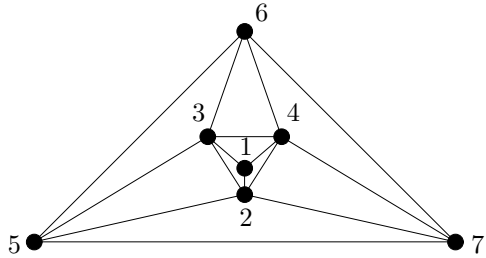
\begin{figure}[ht]
\begin{center}
\begin{picture}(8,4.5) \put(2,-1){
\put(2,5){\p}\put(2.2,5.2){6} \put(1,3.3){3}\put(1.3,3){\p}
\put(2.8,3.3){4}\put(2.7,3){\p}
\put(-2.5,0.8){5}\put(-2,1){\p}\put(6.3,0.8){7}\put(6,1){\p}
\put(1.9,1.3){2}\put(2,1.9){\p} \put(1.9,2.65){1}\put(2,2.4){\p}

\path(2,5)(1.3,3)(2.7,3)(2,5)\path(2,5)(-2,1)(6,1)(2,1.9)(2,2.4)(2.7,3)(6,1)(2,5)
\path(1.3,3)(-2,1)(2,1.9)(1.3,3)(2,2.4)\path(2,1.9)(2.7,3)
}\end{picture} \caption{Co-$(T_2)$ graph}\label{co-T2}
\end{center}
\end{figure}

\begin{theorem} \label{anti-t1}
The graph co-$(T_2)$ is non-representable.
\end{theorem}

\begin{proof}
Assume that the graph in Fig.~\ref{co-T2} is $k$-representable for
some $k$ and $W$ is a word-representant for it. The vertices 1,2,3,4
form a clique; so, their appearances $1^i,2^i,3^i,4^i$ in $W$ must
be in the same order for each $i=1,2,\ldots,k$. By symmetry and
Proposition~\ref{pr1} we may assume that the order is 1234. Now let
$I_1$, $I_2,\ldots,I_k$ be the set of all $[2^i,4^i]$-intervals in
$W$.
Two cases
are possible.
\begin{itemize}
\item[1.] There is an interval $I_j$ such that 7 belongs to it. Then
since 2,4,7 form a clique, 7 must be inside each of the intervals
$I_1$, $I_2,\ldots,I_k$. But then 7 is adjacent to 1, a
contradiction.
\item[2.] 7 does not belong to any of the intervals $I_1$, $I_2,\ldots,I_k$.
Again, since 7 is adjacent to 2 and 4, each pair of consecutive
intervals $I_j$, $I_{j+1}$ must be separated by a single 7. But then
7 is adjacent to 3, a contradiction.
\end{itemize}
\end{proof}

\begin{remark}
Note that the existence of edges between
vertices 5,6,7 was not used in the proof of Theorem~\ref{anti-t1}. So, we
actually have four counter-examples to the converse of
Theorem~\ref{t1}.
\end{remark}

What about other non-representable graphs? Or, rather, which
important classes of graphs are not contained in the class of
representable graphs? We establish the following classes to be not
necessarily representable:
\begin{itemize}
 \item chordal (see the rightmost graph in Fig.~\ref{small}), and
thus perfect,
 \item line (second graph in Fig.~\ref{small}),
 \item co-trees (the graph co-$T_2$ earlier), and thus co-bipartite and co-comparability,
 \item 2-outerplanar (the first and last graphs in Fig.~\ref{small}), and thus
 planar,
  \item split (the last graph in Fig.~\ref{small}),
  \item 3-trees (the last graph in Fig.~\ref{small}), and thus partial 3-trees.
\end{itemize}
On the other hand, the 4-clique, $K_4$, is representable and it
belongs to all these classes.

\paragraph{\bf The effect of graph operations}

One may want to explore which operations on graphs preserve
representability (or non-representability). We pinpoint one such
operation, and list others that are orthogonal.

\begin{enumerate}
\item The following operation on a representable graph yields a
representable graph: Replace any node with a comparability graph,
connecting all the new nodes to the neighbors of the original node.
I.~e., replacing a node with a \emph{module} that is a comparability
graph.

\item A generalization of Proposition~\ref{pr1a} on identifying cliques of size more than 1 from two representable graphs
is false. Indeed, consider the rightmost graph in Fig.~\ref{small}
without a node of degree 2 connected to the end points of edge $e$
(denote this graph by $G$), and identify $e$ with an edge in a
triangle $T$ resulting in obtaining the rightmost graph in
Fig.~\ref{small}. Both $G$ and $T$ are representable, but gluing
them through an edge (a clique of size 2) results in a
non-representable graph.

\item The complement to a non-representable graph can be
permutationally representable: see, for example, the second graph
in~Fig.~\ref{small}.

 \item Not much can be said in general on the taking line graph
 operation. For example, the second non-representable graph
 in~Fig.~\ref{small} is obtained from $K_{2,3}$ together with an
 edge between the nodes of degree 3, which is representable. On the
 other hand, there are many easy constructible examples when
 representable graphs go to representable graphs by taking line graph
 operation.



\end{enumerate}


\section{Concluding Remarks and Open Questions}\label{open}

It is natural to ask about optimization problems on representable
graphs. Theorem~\ref{obs:3col} implies that many classical
optimization problems are NP-hard on representable graphs:

\begin{obs}
The optimization problems Independent Set, Dominating Set, Graph
Coloring, Clique Partition, Clique Covering are NP-hard on
representative graphs.
\end{obs}

Note that it may be relevant whether the representation of the graph as a
semi-transitive digraph is given; solvability under these conditions is open.

However, some problems remain polynomially solvable:

\begin{obs}
The Clique problem is polynomially solvable on representation graphs.
\end{obs}

Indeed, we can simply use the fact
that the neighborhood of any node is a comparability graph. The
clique problem is easily solvable on comparability graphs. Thus, it
suffices to search for the largest clique within all induced
neighborhoods.

\smallskip

There are still many questions one can ask on representable graphs,
some of which are stated below.

\begin{enumerate}
\item Is it NP-hard to decide whether a graph is representable?
\item What is a tighter upper bound on the representation
number of a graph, in terms of $n$? We know that it lies between $n/2$
and $n$.

















\item Can one characterize the forbidden subgraphs of representative graphs? This problem seems to be hard since even for 2-representable (i.~e. circle) graphs such a characterization is unknown.

\end{enumerate}

\end{document}